\documentclass[10pt,a4paper]{article}

\usepackage[final]{optional}

\usepackage{amsfonts, amsmath, wasysym}

\usepackage[usenames]{color}
\usepackage[pdftex]{graphicx}  

\usepackage{amssymb,amsthm,
paralist
}

\usepackage{
latexsym,
}

\usepackage{url}

\definecolor{darkgreen}{rgb}{0,0.5,0}
\definecolor{darkred}{rgb}{0.7,0,0}
\usepackage[colorlinks, 
citecolor=darkgreen, linkcolor=darkred
]{hyperref}

\theoremstyle{plain}
\newtheorem{lemma}{Lemma}[section]
\newtheorem{thm}[lemma]{Theorem}
\newtheorem{prop}[lemma]{Proposition}

\theoremstyle{definition}

\newtheorem{rmk}[lemma]{Remark}

\numberwithin{equation}{section}

\newcommand{\m}{\ensuremath{{\cal M}}}

\newcommand{\pl}[2]{{\frac{\partial #1}{\partial #2}}}

\newcommand{\pll}[2]{{\frac{\partial^2 #1}{\partial #2^2}}}

\newcommand{\al}{\alpha}
\newcommand{\be}{\beta}
\newcommand{\ga}{\gamma}

\newcommand{\Om}{\Omega}

\newcommand{\la}{\lambda}

\newcommand{\si}{\sigma}

\renewcommand{\th}{\theta}

\newcommand{\vph}{\varphi}
\newcommand{\ep}{\varepsilon}

\newcommand{\R}{\ensuremath{{\mathbb R}}}

\newcommand{\Z}{\ensuremath{{\mathbb Z}}}
\newcommand{\C}{\ensuremath{{\mathbb C}}}

\newcommand{\downto}{\downarrow}
\newcommand{\upto}{\uparrow}

\newcommand{\lap}{\Delta}

\DeclareMathOperator{\Vol}{Vol}
\DeclareMathOperator{\VolEx}{VolEx}

\def\blbox{\quad \vrule height7.5pt width4.17pt depth0pt}

\newcommand{\beq}{\begin{equation}}
\newcommand{\eeq}{\end{equation}}
\newcommand{\beqa}{\begin{equation}\begin{aligned}}
\newcommand{\eeqa}{\end{aligned}\end{equation}}
\newcommand{\brmk}{\begin{rmk}}
\newcommand{\ermk}{\end{rmk}}
\newcommand{\partref}[1]{\hbox{(\csname @roman\endcsname{\ref{#1}})}}
\newcommand{\half}{\frac{1}{2}}

\newcommand{\cmt}[1]{\opt{draft}{\textcolor[rgb]{0.5,0,0}{
$\LHD$ #1 $\RHD$\marginpar{\blbox}}}}

\newcommand{\Ric}{{\mathrm{Ric}}}

\newcommand*\dz{\ensuremath{\mathrm{d}z}}

\usepackage{soul}

\title{{
\bf
Uniqueness of Instantaneously Complete Ricci flows
} 
\\ 
\cmt{DRAFT with comments}
}
\author{Peter M. Topping}
\date{7 May 2013} 

\begin{document}

\maketitle
\parskip=10pt

\begin{abstract}
We prove uniqueness of instantaneously complete Ricci flows on surfaces. We do not require any bounds of any form on the curvature or its growth at infinity, nor on the metric or its growth (other than that implied by instantaneous completeness).
Coupled with earlier work, particularly \cite{JEMS, GT2}, this completes the well-posedness theory for instantaneously complete Ricci flows on surfaces.
\end{abstract}

\section{Introduction}
\label{intro}

Hamilton's Ricci flow equation \cite{ham3D}
\beq
\label{RFeq}
\pl{}{t}g(t)=-2\,\Ric[g(t)]
\eeq
for an evolving smooth Riemannian metric $g(t)$ on a manifold $\m$ 
has a short-time existence and uniqueness theory valid for closed $\m$ \cite{ham3D, deturck} and  also for noncompact $\m$ when the initial metric and subsequent flows are taken to be complete and of bounded curvature \cite{shi, chenzhu, kotschwar}. The flow will exist until such time as the curvature blows up.

In this paper we are interested in the case that the domain is two-dimensional (also the simplest setting for K\"ahler Ricci flow) when the results are substantially stronger (e.g. \cite{JEMS, GT2, JMS09} etc.). The first improvement is that the flow can be started with a completely general surface, assumed neither to be complete nor of bounded curvature. It is then possible to flow smoothly forward in time within the class of \emph{instantaneously complete} Ricci flows (i.e. so that $g(t)$ is complete once $t>0$). The second improvement is that the existence time of the flow can be given precisely in terms of the topology and area of the initial surface (normally infinity). In contrast to the higher dimensional theory, the flow will, in general, exist beyond the first time that the curvature blows up \cite{GT4}, and indeed the curvature can be unbounded for all time \cite{GT3}. The third improvement is that the flow enjoys a type of maximality property over all Ricci flows with the same initial surface (\emph{maximally stretched} in the terminology of \cite{GT2}) which already yields a soft uniqueness result. The fourth improvement is that 
in most cases the asymptotics of the flow can be explained in  detail.

In \cite{JEMS} it was proposed that these instantaneously complete Ricci flows should be \emph{unique} (see also \cite[Conjecture 1.5]{GT2}), and partial results in this direction have been given for specific incomplete initial surfaces \cite{JEMS}, for surfaces covered (conformally) by $\C$ \cite{GT2} (exploiting work of Chen \cite{strong_uniqueness}, Yau \cite{Yau73} and Rodriguez, V\'azquez and Esteban \cite{RVE}), for surfaces of uniformly negative curvature (within a restricted class of flows) \cite{GT1} and for flows that are majorised by hyperbolic metrics (within various restricted classes of flows) \cite{GT2, giesen_phd},
in addition to the closed case already covered by Hamilton \cite{ham3D} and the case that the initial metric \emph{and} the flow itself are assumed to be complete and of bounded curvature (\cite[Theorem A.4]{GT2}, or \cite{chenzhu, kotschwar}).
Moreover, Chen \cite[Theorem 3.10]{strong_uniqueness} proves the result under the additional assumptions that the initial metric is complete, of bounded curvature, and non-collapsed.

In this paper we fully settle the uniqueness question for instantaneously complete Ricci flows. They are always unique, without any restriction:

\begin{thm}[Main Theorem]
\label{mainthm}
Let $g(t)$ and $\bar g(t)$ be two instantaneously complete Ricci flows on any surface $\m$, defined for $t\in [0,T]$, with $g(0)=\bar g(0)$. Then $g(t)=\bar g(t)$ for all $t\in [0,T]$.
\end{thm}

Coupled with earlier existence theory, particularly that of Giesen and the author \cite{GT2}, we then have:

\begin{thm}[Incorporating \cite{GT2}]
\label{fullthm}
Let $(\m^2,g_0)$ be any smooth Riemannian surface, possibly incomplete and/or with unbounded curvature. 
Depending on the conformal type, we define $T\in(0,\infty]$ 
by\footnote{Note that in the case that $\m=\mathbb C$, we set $T=\infty$ if $\Vol_{g_0}\mathbb C=\infty$. Here $\chi(\m)$ is the Euler characteristic.}
  \[ T := \begin{cases}
    \frac1{4\pi\chi(\m)}\Vol_{g_0}\m & \text{if }(\m,g_0)\cong\mathcal S^2, \mathbb C\text{ or }\mathbb R\!P^2, \\
    \qquad\infty & \text{otherwise}.
  \end{cases} \]
Then there exists a {\bf unique} smooth Ricci flow $g(t)$ on $\m^2$, defined for $t\in[0,T)$ such that 
  \begin{compactenum} 
  \item $g(0)=g_0$;
  \item $g(t)$ is instantaneously complete.
  \end{compactenum}
In addition, this Ricci flow $g(t)$ is maximally stretched (see Remark \ref{max_stretch}).
If $T<\infty$, then we have 
\[
\Vol_{g(t)}\m = 4\pi\chi(\m) (T-t)\longrightarrow\quad 0 \quad\text{ as } t\nearrow T, \] 
and in particular, $T$ is the maximal existence time.
\end{thm}

\begin{rmk}
\label{max_stretch}
Recall from \cite{GT2} and \cite[Remark 1.5]{GT4} that the \emph{maximally stretched}  assertion means that $g(t)$ lies `above' any another Ricci flow with the same or lower initial data.
More precisely,
if $0\leq a<b\leq T$ and $\tilde g(t)$ is any Ricci flow on $\m$ for $t\in [a,b)$ with $\tilde g(a)\leq g(a)$ (with $\tilde g(t)$ not necessarily complete or of bounded curvature) then $\tilde g(t)\leq g(t)$ for every $t\in[a,b)$. 
\end{rmk}

If one relaxes the notion in which the Ricci flow should attain its initial data, then non-uniqueness \emph{can} occur \cite{revcusp}.

One can consider Theorem \ref{fullthm} also as the well-posedness theorem for the logarithmic fast diffusion equation, a topic that has been studied extensively (e.g. \cite{DdP95, DD},
\cite[\S 3.2]{DK} and \cite[\S 8.2]{vaz}).
Ricci flow on surfaces $\m$ preserves the conformal class of the metric and can be written
$$\pl{}{t}g(t)=-2K. g(t),$$
where $K$ is the Gauss curvature of $g(t)$.
In this case, we may take local isothermal coordinates
$x$ and $y$, and write the flow 
$g(t)=e^{2u}|\dz|^2:=e^{2u}(dx^2+dy^2)$
for some locally-defined scalar time-dependent function $u$ which will then satisfy the local equation
\beq
\label{2DRFeq}
\pl{u}{t}=e^{-2u}\lap u = -K.
\eeq
The area density $U:=e^{2u}$ then satisfies
\beq
\label{2DRFeq2}
\pl{U}{t}=\lap \log U.
\eeq
This equation has been heavily studied in the case that the domain is $\R^2$ (e.g. \cite{DdP95, DD, vaz, DK})
with good existence theory, but troublesome uniqueness. This paper demonstrates that instantaneous completeness is the right substitute for a growth (or other) condition in this and all other cases.
We note that the well-known result of 
Rodriguez, V\'azquez and Esteban \cite{RVE} that in the special case of the domain being $\R^2$, the maximally stretched solution is uniquely characterised by the growth condition 
\begin{equation}
\label{growth_cond}
\frac{1}{U}\leq C\left(\frac{r^2(\log r)^2}{2t}\right)\quad
\text {for some constant } C<\infty,
\end{equation}
is subsumed in the statement of Theorem \ref{mainthm}, as it was in \cite{GT2} (see \cite[\S 3.2]{GT2}, where the ideas of \cite{RVE} were combined in particular with Yau's Schwarz lemma \cite{Yau73}). The conformal factor $2t/(r\log r)^2$ here is representing an expanding hyperbolic cusp, and is therefore very natural, although it could be weakened further using 
Theorem \ref{mainthm}.

Although Theorem \ref{fullthm} handles all surfaces as possible initial data, it may be instructive to consider the special case that the initial surface is the flat two-dimensional disc $(D,g_0)$, which corresponds to the initial conformal factor $u\equiv 0$. It is a spurious feature of this example that it has zero curvature, so one Ricci flow continuation would be the static solution. However, even amongst solutions $u$ to \eqref{2DRFeq} that are continuous up to the boundary $\partial D$, the standard theory of parabolic equations tells us that we are free to prescribe $u$ on 
$\partial D$ (that then giving uniqueness in this restricted class).
Theorem \ref{fullthm} tells us that there is another way of flowing, where we send in an infinite amount of area at spatial infinity to make the metric immediately complete. The flow will be smooth all the way down to $t=0$, so for small time the flow will look a little like a flat disc on the interior.
However, in a boundary layer, the metric will look like the Poincar\'e metric on the disc, scaled to have curvature $-\frac{1}{2t}$.
By exploiting some deep ideas of Perelman \cite[\S 10]{P1} it can be shown that the thickness of the layer where the flow is far from the flat disc is controlled by $\sqrt{t}$.

Theorem \ref{mainthm} implies in particular that there is only one way we can send in area from infinity in order to make the surface complete. If we try to send in less, then the metric will not be complete. If we try to send in more, then the nonlinearity acts as a damping mechanism, and nothing reaches the interior.
Theorem \ref{mainthm} implies that these effects are exactly 
complementary.

\section{Applications of Uniqueness}

The uniqueness of Theorem \ref{mainthm} serves to finish the well-posedness theory for this equation, and highlights the Ricci flow constructed in \cite{GT2} as the natural one to study both in general and as a tool to understand spectral and dynamical properties of surfaces (cf. \cite{AAR} for example) and their uniformisation. As a by-product, it also immediately implies that in the special case that the initial surface is complete and of bounded curvature, then the Ricci flow from Theorem \ref{fullthm} coincides with the now-standard Ricci flow constructed by Shi \cite{shi}, 
as originally proved in \cite[Theorem 1.8]{GT2}
(although Shi's flow will not exist for as long \cite{GT4}).

One application of uniqueness is a comparison principle that can allow us to compare Ricci flows whose behaviour on the `boundary at infinity' is unknown.
For example, whenever we have an instantaneously complete Ricci flow $g(t)$, perhaps a Ricci soliton flow that we have written down explicitly, maybe with unbounded curvature, then it must be maximally stretched by Theorem \ref{fullthm} (as a result of Theorem \ref{mainthm}), and all other Ricci flows starting at or below $g(0)$, possibly incomplete and/or with unbounded curvature, must lie below $g(t)$, however they behave `at infinity'.

A further application is that 
it is often useful to know that an instantaneously complete Ricci flow that has arisen in a construction (for example as a limit of other Ricci flows) is the flow that would have arisen from the existence theory using the same initial metric, and Theorem \ref{mainthm} provides this for us.
For example, the argument in \cite{HCT} would have been simplified 
by knowing that a Ricci flow of possibly unbounded curvature that starts at a complete soliton metric is necessarily the corresponding soliton flow while it is known to be complete. Theorem \ref{mainthm} guarantees that it is.

Finally we remark that the estimates from this paper imply stability results for two-dimensional Ricci flow.

\section{Proof of the uniqueness of Theorem \ref{mainthm}}

We are free to assume that the Ricci flow $\bar g(t)$ of Theorem \ref{mainthm} is a maximally stretched solution arising from \cite[Theorem 1.3]{GT2} (i.e. from the existence part of Theorem \ref{fullthm}) and that $g(t)$ is any other instantaneously complete Ricci flow with $g(0)=\bar g(0)$.
By the Uniformisation theorem, the universal cover of $\m$ must be conformally equivalent to $S^2$, $\C$ or $D$, and it suffices to prove that the lifts of $g(t)$ and $\bar g(t)$ to the universal cover agree. In the case that the universal cover is $S^2$, then we are in the closed case that was dealt with originally by Hamilton \cite{ham3D} (see also \cite{formations}). In the case that the universal cover is $\C$, Giesen and the author already proved the result \cite[\S 3.2]{GT2} based on an estimate of Rodriguez, V\'azquez and Esteban \cite{RVE}. It therefore suffices to assume that $\m=D$, the two-dimensional disc.

Since we know that $g(t)\leq \bar g(t)$ by the maximally stretched property, and $g(0)=\bar g(0)$, we need only show that relative to $g(t)$, the flow $\bar g(t)$ cannot \emph{lift off} as time increases. The key result is the following quantitative estimate that 
considers the area measured with respect to $g(t)$ of a smaller (concentric) disc 
$D_r:=\{z\in\C\ :\ |z|<r\}\subset D$, and constrains how much larger it could be if we were allowed to make the metric arbitrarily larger on the boundary of a larger disc $D_R\subset D$, where
$r<R<1$.

\begin{lemma}[Interior area estimate]
\label{keylemma}
Fix $r_0\in (0, 1)$ and $\gamma\in (0,\half)$. Then for 
$R\in (r_0,1)$ sufficiently close to $1$, depending on $r_0$, the following holds.
Suppose $g(t)$ is any instantaneously complete Ricci flow on $D$, for $t\in [0,T]$, and $G(t)$ is any Ricci flow on $\overline {D_R}$ with the properties that $G(0)=g(0)$ and $G(t)\geq g(t)$
on $\overline {D_R}$ for $t\in [0,T]$.
Then for all $t\in [0,T]$, we have
\begin{equation}
\label{mainest}
\Vol_{g(t)}D_{r_0}\leq \Vol_{G(t)}D_{r_0}\leq \Vol_{g(t)}D_{r_0}+\frac{C(\ga,r_0)t}{[-\log(-\log R)]^\ga}.
\end{equation}
\end{lemma}

In practice we will take $G(t)$ to be the restriction to $\overline{D_R}$ of the maximally stretched Ricci flow $\bar g(t)$ on the entire disc $D$ with $\bar g(0)=g(0)$. In that case, if we fix $r_0$ and $\ga$ as in the lemma, and apply the lemma for larger and larger $R\in (r_0,1)$ we obtain in the limit that
$$\Vol_{\bar g(t)}D_{r_0}= \Vol_{g(t)}D_{r_0},$$
from which we deduce that $g(t)=\bar g(t)$ on $D_{r_0}$ because $\bar g(t)\geq g(t)$. By taking the limit $r_0\nearrow 1$, we find that $g(t)=\bar g(t)$ throughout $D$
for all $t\in [0,T]$, which completes the proof of Theorem \ref{mainthm}, modulo Lemma \ref{keylemma}.

The rest of the paper is devoted to proving Lemma \ref{keylemma}. Indeed we will prove a slightly stronger statement that we give in Lemma \ref{stronger_lemma}. The basic technique of the proof of these estimates is to define a weighted area consisting of the integral of a certain cut-off function $\vph:D\to [0,1]$ with support in $D_R$, and identically equal to $1$ on $D_{r_0}$, and control the evolution of the difference of the weighted areas defined with respect to $g(t)$ and $G(t)$, in a manner reminiscent of the work of 
Rodriguez, V\'azquez and Esteban \cite{RVE} mentioned above.
Great care is required to choose the cut-off $\vph$ in order to measure the right flux of area in our case.


\section{Proof of the interior area estimate, Lemma \ref{keylemma}}
\label{proof_interior}

We begin the proof working under the hypotheses of Lemma \ref{keylemma},
\emph{except} that we do not assume for the moment that $G(0)=g(0)$.

In order to respect the geometry of the situation, and dramatically reduce the number of logarithms in computations, we will view the (punctured) disc conformally as a half-cylinder 
$(0,\infty)\times S^1$, and take (isothermal) logarithmic polar coordinates $(s,\th)$ instead of the standard Cartesian isothermal coordinates $(x,y)$, where $s=-\log r=-\half \log (x^2+y^2)$,
i.e. $x+iy=e^{-(s+i\th)}$ for $\th\in S^1=\R/(2\pi\Z)$.

In these coordinates, the conformal factor of the complete hyperbolic metric $g_{poin}=H(ds^2+d\th^2)$ on 
$D$ is given by 
\begin{equation}
\label{Hdef}
H(s)=\frac{1}{\sinh^2 s}
\end{equation}
for $s>0$.
The so-called \emph{big-bang} Ricci flow is then given by
$2t\,g_{poin}$ for $t>0$.
We will also use these coordinates to write the Ricci flows 
$g(t)=U(ds^2+d\th^2)$ and $G(t)=V(ds^2+d\th^2)$ for $s>0$ and 
$s>S:=-\log R$ respectively, where the scalar functions $U$ and $V$ will then satisfy \eqref{2DRFeq2}. By hypothesis, we have $U\leq V$.
Moreover, from \cite[Lemma 2.1(i)]{GT2} (and \cite{GT1}, based on a combination of Chen's scalar curvature estimates \cite{strong_uniqueness} and the Schwarz-Pick-Ahlfors-Yau lemma) we know that the big-bang Ricci flow lies below any instantaneously complete Ricci flow, i.e. $2t\,g_{poin}\leq g(t)$ for $t>0$. Thus we have
\begin{equation}
\label{HUV}
2t H\leq U\leq V.
\end{equation}

We are free to make the assumption that $r_0\in (\half,1)$, 
i.e. $s_0:=-\log r_0\leq \log 2$, since reducing $r_0$ only makes \eqref{mainest} weaker, and we also assume that
$R\geq r_0^{1/3}$, i.e. that 
\begin{equation}
\label{Sassump}
S:=-\log R\leq \frac{1}{3}(-\log r_0)
=:\frac{1}{3}s_0.
\end{equation}
Because $\sinh(s)$ is convex for $s>0$, and hence $\sinh(s)\leq \sinh(\log 2)\frac{s}{\log 2}=\frac{3s}{4\log 2}$
for $s\in (0,\log 2)$, 
the assumption $s_0\leq \log 2$ allows us to combine \eqref{Hdef}
and \eqref{HUV} to give
\begin{equation}
\label{H2}
U^{-1}\leq \frac{Cs^2}{t}
\end{equation}
for $s\in (0,s_0)$ and universal $C$.

In order to prove Lemma \ref{keylemma}, we will consider the evolution of the difference of time-dependent weighted areas
\begin{equation}
\label{Jeq1}
\begin{aligned}
J(t)&:=\int_{D_R} \vph d\mu_{G(t)}-
\int_{D_R} \vph d\mu_{g(t)}\\
&=\int_{S^1}\int_{S}^\infty (V-U)\vph\, ds\,d\th,
\end{aligned}
\end{equation}
for an appropriate rotationally symmetric cut-off function $\vph:D\to [0,1]$ with support in $D_R$, and which is identically equal to $1$ on $D_{r_0}$. We will normally view $\vph$ as a function of $s$.
Later we will arrange, amongst other things, that $\vph(s)$ is smooth away from $s=S$ and $s=s_0/2$, and $C^1$ for all $s>0$.
Later when we assume that $g(0)=G(0)$, we will have $J(0)=0$. Thus Lemma \ref{keylemma}
reduces to controlling the quantity
$$J(t)\geq 
\int_{D_{r_0}} d\mu_{G(t)}-\int_{D_{r_0}} d\mu_{g(t)}=
\Vol_{G(t)}D_{r_0}-\Vol_{g(t)}D_{r_0}.$$
In order to do this, we differentiate \eqref{Jeq1} in time, using 
\eqref{2DRFeq2} to obtain
\begin{equation}
\label{dJdt1}
\begin{aligned}
\frac{dJ}{dt}&=
\int_{S^1}\int_{S}^\infty \pl{}{t}(V-U)\vph\, dsd\th\\
&=\int_{S^1}\int_{S}^\infty \lap(\log V-\log U)\vph\, dsd\th\\
&=\int_{S^1}\int_{S}^\infty (\log V-\log U)\vph''(s)\, dsd\th,
\end{aligned}
\end{equation}
where $\lap=\pll{}{s}+\pll{}{\th}$ and the integration by parts is valid because $V$ and $U$ have the same asymptotics at $s=\infty$.
Now, for $\la\in(0,1)$ and $x\geq 0$, we know that 
$\log(1+x)\leq \frac{1}{\la}x^\la$, and therefore with
$\la=\frac{\ga}{1+\ga}\in (0,\frac{1}{3})$ we have
$$\log V-\log U=\log(1+(V-U)/U)\leq 
\frac{1+\ga}{\ga}\left[\frac{V-U}{U}\right]^\frac{\ga}{1+\ga},$$
which we may insert  into \eqref{dJdt1}.
If we make the additional assumption that 
$\vph''(s)\leq 0$ for 
$s\in (s_0/2,s_0)$ (as we shall arrange below) then we have
\begin{equation}
\label{dJdt1b}
\frac{dJ}{dt}
\leq C(\ga)\int_{S^1}\int_{S}^{s_0/2} 
\left[\frac{V-U}{U}\right]^\frac{\ga}{1+\ga}
|\vph''(s)|\, dsd\th.
\end{equation}
and applying H\"older's inequality, we find that
\begin{equation}
\label{dJdt2}
\begin{aligned}
\frac{dJ}{dt}&\leq
C(\ga)
\left(\int_{S^1}\int_{S}^{s_0/2} 
[V-U]\vph\, dsd\th\right)^{\frac{\ga}{1+\ga}}
\left(\int_{S^1}\int_{S}^{s_0/2} 
U^{-\ga}|\vph''|^{1+\ga}\vph^{-\ga}\, dsd\th
\right)^{\frac{1}{1+\ga}}\\
&\leq C(\ga) J^\frac{\ga}{1+\ga}
\left(\int_{S}^{s_0/2} 
\left(\frac{s^2}{t}\right)^{\ga}|\vph''|^{1+\ga}\vph^{-\ga}\, ds
\right)^{\frac{1}{1+\ga}},
\end{aligned}
\end{equation}
where we have used \eqref{Jeq1}, \eqref{HUV} and \eqref{H2} in the last inequality.
Defining 
\begin{equation}
\label{Qdef}
Q:=\int_{S}^{s_0/2} 
s^{2\ga}|\vph''|^{1+\ga}\vph^{-\ga}\, ds,
\end{equation}
we  then have the main differential inequality
\begin{equation}
\label{mainODI}
\frac{d}{dt}J^{\frac{1}{1+\ga}}\leq C(\ga)
t^{-\frac{\ga}{1+\ga}}Q^\frac{1}{1+\ga}.
\end{equation}
We would like now to carefully choose the cut-off function $\vph$ so that we can be sure of good bounds on $Q$. This will allow us to integrate \eqref{mainODI} to finish the proof.
The function $\vph$ will be derived from the following specific function for which we compile some elementary facts.

\begin{prop}[Flux function]
\label{fluxprop}
For $a\in (0,1)$, the function $f:[a,1]\to\R$ defined by
\begin{equation}
\label{fdef}
f(\si)=\frac{1}{(-\log a)}\left[\si (\log \si-\log a)-(\si-a)\right]
\end{equation}
satisfies
\begin{compactenum}[(i)]
\item
$f(a)=0$,
\item
$\displaystyle f(1)=1-\frac{1-a}{(-\log a)}\in (0,1)$,
\item
$\displaystyle f'(\si)=\frac{1}{(-\log a)}\left[\log \si-\log a\right]\geq 0$,
\item
$f'(a)=0$,
\item
$f'(1)=1$,
\item
$\displaystyle f''(\si)=\frac{1}{\si(-\log a)}\geq 0$.
\end{compactenum}
In particular, we can extend $f$ to a $C^1$ function $f:\R\to [0,1]$, smooth except at $\si=a$ and $\si=1$, with the properties that 
\begin{compactenum}[(a)]
\item
$f(\si)=0$ for all $\si\leq a$,
\item
$f(\si)=1$ for all $\si\geq 2$, 
\item
$f''(\si)\leq 0$ for all $\si\in (1,2]$.
\end{compactenum}
\end{prop}

\begin{figure}[h]
\centering
\includegraphics{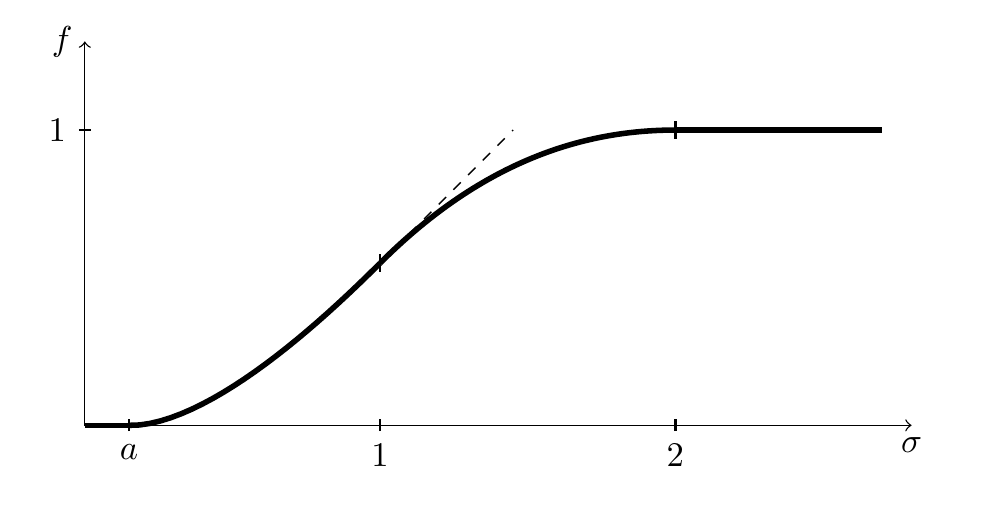}


\caption{The Flux function $f$.}
\end{figure}

The $C^1$ cut-off function $\vph: (0,\infty)\to [0,1]$ will arise as an appropriate scaling of $f$. Indeed, 
for our $S:=-\log R>0$ and $s_0:=-\log r_0<\infty$, with our standing assumption that $S\leq\frac{s_0}{3}$, we can then set $a=\frac{2S}{s_0}\in (0,1)$ and define
%
%
$$\vph(s)=f\left(\frac{2s}{s_0}\right),$$
which is smooth except at $s=S$ and $s=s_0/2$, and satisfies
\begin{compactenum}[(i)]
\item
$\vph(s)=0$ for $s\in (0,S]$,
\item
$\vph(s)=1$ for $s\geq s_0$,
\item
$\vph''(s)\leq 0$ for all $s\in (s_0/2,s_0)$.
\end{compactenum}

With this choice of $\vph$, we will be able to control $Q$ from \eqref{Qdef} according to:
\begin{prop}
\label{Qprop}
$$Q\leq \frac{C(\ga)}{
(-\log r_0)\left[\log(-\log r_0)-\log(-\log R)\right]^\ga}
$$
\end{prop}

Accepting this proposition for a moment, we can then integrate 
\eqref{mainODI} to give
\begin{equation*}
J^{\frac{1}{1+\ga}}(t)-J^{\frac{1}{1+\ga}}(0)
\leq C(\ga)
t^{\frac{1}{1+\ga}}
\left[\frac{1}{
(-\log r_0)\left[\log(-\log r_0)-\log(-\log R)\right]^\ga}
\right]^\frac{1}{1+\ga},
\end{equation*}
and we have proved:

\begin{lemma}
\label{stronger_lemma}
Suppose $r_0\in (\half, 1)$,
$R\in (r_0^{1/3},1)\subset(r_0,1)$ and $\gamma\in (0,\half)$.
Suppose $g(t)$ is any instantaneously complete Ricci flow on $D$, for $t\in [0,T]$, and $G(t)$ is any Ricci flow on $\overline {D_R}$ with the property that $G(t)\geq g(t)$
on $\overline {D_R}$ for $t\in [0,T]$.
Then for all $t\in [0,T]$, we have
\begin{equation}
\begin{aligned}
\left[\Vol_{G(t)}D_{r_0}-\Vol_{g(t)}D_{r_0}\right]^{\frac{1}{1+\ga}}
&\leq
\left[\Vol_{G(0)}D_{R}-\Vol_{g(0)}D_{R}\right]^{\frac{1}{1+\ga}}\\
&\quad+C(\ga)
\left[\frac{t}{
(-\log r_0)\left[\log(-\log r_0)-\log(-\log R)\right]^\ga}
\right]^{\frac{1}{1+\ga}}
.
\end{aligned}
\end{equation}
\end{lemma}

In the setting of Lemma \ref{keylemma}, we also know that $G(0)=g(0)$, and so Lemma \ref{stronger_lemma} immediately implies Lemma \ref{keylemma}.

\section{Estimating $Q$}

It remains to prove Proposition \ref{Qprop}, estimating $Q$.
We begin by changing coordinates in \eqref{Qdef}, setting 
$\si=\frac{2s}{s_0}$ (keeping $a=\frac{2S}{s_0}$) so that
\begin{equation}
\begin{aligned}
\label{Q}
Q&=
\int_{a}^{1} 
\left(\frac{s_0\si}{2}\right)^{2\ga}\left(\left(\frac{s_0}{2}\right)^{-2}|f''(\si)|\right)^{1+\ga}f(\si)^{-\ga}\left( \frac{s_0}{2}\right)d\si\\
&=
\frac{2}{s_0}
\int_{a}^{1} 
\si^{2\ga}|f''|^{1+\ga}f^{-\ga}\,d\si\\
&=
\frac{2}{s_0}
\int_{a}^{1} 
\si^{\ga-1}(-\log a)^{-1}\left[\si (\log \si-\log a)-(\si-a)\right]^{-\ga}\,d\si,
\end{aligned}
\end{equation}
where we have used the definition of $f$ in \eqref{fdef} and Part (vi) of Proposition \ref{fluxprop}.
In order to proceed further, we must split the integral into a part on which $\si$ is of the order of $a$ (let us say for $\si\in (a,e^2a)$), and the remaining part where $\si$ is bounded away from $a$ (i.e. $\si\in (e^2a,1)$). Here we are assuming that $e^2a<1$; if not, we can estimate the whole range $\si\in (a,1)$ in one go, and the computation below simplifies a lot.

To deal with the range $\si\in (e^2a,1)$, we change variables from $\si$ to $\al:=\log\frac{\si}{a}$, so $\al\in (2,-\log a)$. Within this range we can estimate
\begin{equation*}
\begin{aligned}
\si (\log \si-\log a)-(\si-a)
&\geq\frac{\si}{2}(\log \si-\log a)
+\frac{\si}{2}\left(\log (e^2a)-\log a\right)-(\si-a)\\
&=\frac{\si}{2}\left(\log \frac{\si}{a}\right)+a\\
&\geq \frac{\si}{2}\left(\log \frac{\si}{a}\right).
\end{aligned}
\end{equation*}
Therefore, 
\begin{equation}
\begin{aligned}
\label{Q1}
Q_1&:=\frac{2}{s_0}
\int_{e^2a}^{1} 
\si^{\ga-1}(-\log a)^{-1}\left[\si (\log \si-\log a)-(\si-a)\right]^{-\ga}\,d\si\\
&\leq \frac{2}{s_0}\int_{e^2a}^{1} 
\si^{\ga-1}(-\log a)^{-1}
\left[
\frac{\si}{2}\left(\log \frac{\si}{a}\right)
\right]^{-\ga}\,d\si\\
&\leq
\frac{C(\ga)}{s_0}
(-\log a)^{-1}
\int_{e^2a}^{1} 
\si^{-1}
\left[\log \frac{\si}{a}
\right]^{-\ga}\,d\si\\
&\leq\frac{C(\ga)}{s_0}
(-\log a)^{-1}
\int_{2}^{-\log a} 
\al^{-\ga}\,d\al\\
&\leq\frac{C(\ga)}{s_0}
(-\log a)^{-\ga}.
\end{aligned}
\end{equation}

Meanwhile, to deal with the range $\si\in (a,e^2a)$, we recall that
for $x\in (-1,0]$, we have $\log(1+x)\leq x-\frac{x^2}{2}$, and therefore
$$\log a-\log \si
=\log \left(1+(\frac{a}\si-1)\right)\leq \left(\frac{a}\si-1\right)-\half\left(\frac{a}\si-1\right)^2$$
and hence
$$
\left[\si (\log \si-\log a)-(\si-a)\right]\geq \frac{1}{2\si}(\si-a)^2
$$
for $\si>a$. We may therefore estimate
\begin{equation}
\begin{aligned}
\label{Q2}
Q_2&:=
\frac{2}{s_0}
\int_{a}^{e^2a} 
\si^{\ga-1}(-\log a)^{-1}\left[\si (\log \si-\log a)-(\si-a)\right]^{-\ga}\,d\si\\
&\leq
\frac{C(\ga)}{s_0}
\int_{a}^{e^2a} 
a^{\ga-1}(-\log a)^{-1}\left[
\frac{1}{\si}(\si-a)^2
\right]^{-\ga}\,d\si\\
%
&= 
\frac{C(\ga)}{s_0}(-\log a)^{-1}
\int_{1}^{e^2} 
\be^{\ga}(\be-1)^{-2\ga}d\be\\
&\leq
\frac{C(\ga)}{s_0}(-\log a)^{-1}
\end{aligned}
\end{equation}
where we have changed variables from $\si$ to $\be=\frac{\si}{a}$, and used the fact that $\ga\in (0,\half)$.
Combining \eqref{Q}, \eqref{Q1} and \eqref{Q2}, we find that
\begin{equation}
\label{almostthere}
\begin{aligned}
Q&=Q_1+Q_2\leq \frac{C(\ga)}{s_0}\left[(-\log a)^{-1}
+(-\log a)^{-\ga}\right]\\
&\leq \frac{C(\ga)}{s_0}(-\log a)^{-\ga}
\end{aligned}
\end{equation}
because by \eqref{Sassump} we know that $-\log a\geq -\log \frac{2}{3}$.
Finally, we appeal again to \eqref{Sassump} to see that 
$\log{s_0}-\log{S}\geq \log 3$, and thus
$$-\log a=\log s_0-\log S-\log2\geq\left(1-\frac{\log 2}{\log 3}\right)\left(\log s_0-\log S\right).$$
When combined with \eqref{almostthere} we conclude that
$$Q\leq 
\frac{C(\ga)}{s_0}(\log s_0-\log S)^{-\ga},$$
as claimed in the proposition.

\section{Further refinements}

The same ideas as above also give the following improvement of 
Lemma \ref{stronger_lemma} in which we do not assume that the flows are ordered. 

\begin{lemma}
\label{even_stronger_lemma}
Suppose $r_0\in (\half, 1)$,
$R\in (r_0^{1/3},1)\subset(r_0,1)$ and $\gamma\in (0,\half)$.
Suppose $g(t)$ is any instantaneously complete Ricci flow on $D$, for $t\in [0,T]$, and $G(t)$ is any Ricci flow on $\overline {D_R}$. 
Defining, for $\Om\subset D_R$ the \emph{volume excess}
$$\VolEx_t \Om:=\Vol_{G(t)}\Om_t-\Vol_{g(t)}\Om_t,$$
where $\Om_t:=\{x\in\Om\ |\ G(t)\geq g(t)\text{ at }x\}$, 
we have for all $t\in [0,T]$ that
\begin{equation}
\begin{aligned}
\left[\VolEx_t D_{r_0}\right]^{\frac{1}{1+\ga}}
&\leq
\left[\VolEx_0 D_{R}\right]^{\frac{1}{1+\ga}}\\
&\quad+C(\ga)
\left[\frac{t}{
(-\log r_0)\left[\log(-\log r_0)-\log(-\log R)\right]^\ga}
\right]^{\frac{1}{1+\ga}}.
\end{aligned}
\end{equation}
\end{lemma}

The only alteration to the proof in Section \ref{proof_interior} is to replace the definition of $J(t)$ in \eqref{Jeq1} by 
$$J(t):=\int_{S^1}\int_{S}^\infty (V-U)_+\vph\, ds\,d\th,$$
where $x_+:=\max\{x,0\}$.
One of the equalities in \eqref{dJdt1} becomes an inequality, but
we still obtain \eqref{mainODI} for this new $J$.

This improvement of the main estimate, together with the directly analogous estimates on $\C$ (as given in \cite{GT2} and \cite{RVE}) and $S^2$ (which is trivial because we can set $\vph\equiv 1$), gives the following alternative to Theorem \ref{mainthm}, which 
is the key to resolving the uniqueness issue for Ricci flows on general fixed surfaces starting with certain Alexandrov spaces, 
following on from Richard \cite{TR}, as we shall explain elsewhere \cite{CT}.

\begin{thm}
Suppose that $g(t)$ and $\bar g(t)$ are two conformally equivalent complete Ricci flows on a surface $\m$, defined for $t\in (0,T]$, both with curvature uniformly bounded from below, and with
\begin{equation}
\label{vol_hyp}
\lim_{t\downto 0}\Vol_{g(t)}\Om
=\lim_{t\downto 0}\Vol_{\bar g(t)}\Om<\infty
\end{equation}
for all $\Om\subset\subset\m$.
Then $g(t)=\bar g(t)$ for all $t\in (0,T]$.
\end{thm}

\begin{proof}
By lifting both flows to the universal cover of $\m$, we may assume that $\m$ is either $S^2$, $D$ or $\C$. The $S^2$ case is effectively covered by the work of Richard \cite{TR}. The $D$ case will be given below, using Lemma \ref{even_stronger_lemma}. The $\C$ case will then be identical, except that one replaces Lemma \ref{even_stronger_lemma} by the analogous estimates from \cite{GT2} based on the estimate of Rodriguez, V\'azquez and Esteban \cite{RVE}, and we omit it.

Let us write $g(t)=U(dx^2+dy^2)$ and $\bar g(t)=V(dx^2+dy^2)$ on the disc $D$. By scaling both flows (multiplying each by a large positive constant, and scaling time by the corresponding factor) we may assume that the curvatures of each flow are bounded below by $-1$. By \eqref{2DRFeq}, we then know that the conformal factors cannot increase very fast, i.e.
$$\pl{U}{t}\leq 2U,$$ 
and $e^{-2t}U$ is monotonically decreasing in $t$.
By \eqref{vol_hyp}, for all $\Om\subset\subset D$, we know that $U$ and thus $e^{-2t}U$ are uniformly bounded in $L^1(\Om)$ as $t\downto 0$, and thus by the monotone convergence theorem, we have  convergence of 
$e^{-2t}U$ and thus $U$, in $L^1(\Om)$ to some limit function 
$W\in L^1_{loc}(D)$ as $t\downto 0$. The same argument may be applied to $V$ to give, by \eqref{vol_hyp}, the same limit $W$.
In particular, we have 
\begin{equation}
\label{VUL1}
\int_\Om |V-U|dxdy \to 0\qquad\text{ as }t\downto 0.
\end{equation}

We may now apply Lemma \ref{even_stronger_lemma} to the flows
$g(t+\ep)$ and $\bar g(t+\ep)$ for $\ep\in (0,T)$. 
This tells us that
for all $r_0\in (\half, 1)$, $R\in (r_0^{1/3},1)\subset(r_0,1)$,
$\gamma\in (0,\half)$ and $0<\ep<t$, we have
\begin{equation}
\begin{aligned}
\lefteqn{\left[\int_{D_{r_0}}|V-U|_+dxdy\bigg|_{t} \right]^{\frac{1}{1+\ga}}
-\left[\int_{D_{R}}|V-U|_+dxdy\bigg|_{\ep} \right]^{\frac{1}{1+\ga}}}&\\
&\hspace{3cm}=
\left[\VolEx_{t-\ep} D_{r_0}\right]^{\frac{1}{1+\ga}}
-\left[\VolEx_0 D_{R}\right]^{\frac{1}{1+\ga}}\\
&\hspace{3cm}\leq C(\ga)
\left[\frac{t-\ep}{
(-\log r_0)\left[\log(-\log r_0)-\log(-\log R)\right]^\ga}
\right]^{\frac{1}{1+\ga}}.
\end{aligned}
\end{equation}
By \eqref{VUL1}, when we let $\ep\downto 0$, we obtain 
$$\int_{D_{r_0}}|V-U|_+dxdy\bigg|_{t}
\leq 
\frac{C(\ga)\,t}{
(-\log r_0)\left[\log(-\log r_0)-\log(-\log R)\right]^\ga}
.$$
We may now take the limit $R\upto 1$ and then $r_0\upto 1$
to conclude that $V\leq U$. By repeating the same argument with $U$ and $V$ switched, we find that $g(t)=\bar g(t)$ as desired.
\end{proof}

\emph{Acknowledgements:} Thanks to Gregor Giesen for conversations on this topic. This work was supported by EPSRC grant number EP/K00865X/1.

{\sc mathematics institute, university of warwick, coventry, CV4 7AL,
uk}\\
\url{http://www.warwick.ac.uk/~maseq}
\end{document}